\documentclass[11pt, reqno]{amsart}
\usepackage[margin=1.1in]{geometry}
\usepackage{amsmath, amssymb, amsthm}

\usepackage{color}

\usepackage[pdftex, colorlinks=true, urlcolor=blue, linkcolor=blue, citecolor=blue, pdfstartview=FitH]{hyperref}

\newtheorem{theorem}{Theorem}
\newtheorem{lemma}{Lemma}
\newtheorem{proposition}{Proposition}

\newtheorem{corollary}{Corollary}

\newtheorem{remark}{Remark}
\newtheorem{example}{Example}

\everymath{\displaystyle}

\title[Discrete and smooth isoperimetric inequalities]{A unified approach to higher order discrete and smooth isoperimetric inequalities}

\author{Kwok-Kun Kwong}
\address{Mathematical Sciences Institute, Australian National University, Canberra, ACT 2601, Australia}
\email{kwok-kun.kwong@anu.edu.au}

\numberwithin{equation}{section}

\begin{document}

\maketitle

\begin{abstract}
We present a unified approach to derive sharp isoperimetric type inequalities of arbitrary high order. In particular, we obtain (i) sharp high order discrete polygonal isoperimetric type inequalities, (ii) sharp high order isoperimetric type inequalities for smooth curves with both upper and lower bounds for the isoperimetric deficit, and (iii) sharp higher order Chernoff type inequalities involving a generalized width function and higher order locus of curvature centers. Our approach involves obtaining higher order discrete or smooth Wirtinger inequalities via Fourier analysis, by examining a family of linear operators. The key to our approach is identifying the appropriate linear operator and translating the analytic inequalities into geometric ones.
\end{abstract}

\section{Introduction}
This paper presents a unified approach to obtaining sharp higher order isoperimetric-type inequalities and their stability, which can be divided into two parts: discrete and smooth. The approach involves identifying the appropriate function space and family of linear operators, analyzing these operators using the corresponding Fourier analysis to derive sharp higher order Wirtinger type inequalities, and translating the analytic information into geometric ones. These results are novel, even in the smooth case. Notably, we are not aware of any prior work that establishes geometric inequalities involving the ``derivatives'' of the ``curvature'' of a polygon.

In the first part of this paper, we derive a series of sharp higher order discrete Wirtinger inequalities together with their stability results, and use them to obtain sharp bounds for the isoperimetric inequalities for polygons. There is already a large number of results on the higher order Wirtinger type inequalities (see for example \cite{alzer1991converses, lunter1994new, milovanovic1997discrete}). However, up to now, the applications of the Wirtinger inequalities to isoperimetric type inequalities seem to be restricted only to the first order case, and these inequalities involve only the area, the length (or sum of the squares of the side lengths), and some special distances, but not the ``curvature'' or its higher ``derivatives'', see for example \cite{block1957discrete, tang1991discrete, zhang1997bonnesen}, or \cite{osserman1979bonnesen, Osserman1978, zhou2011some} for some analogous results in the smooth case. This is not surprising because for a polygon, the area and various lengths or distances involve only the zeroth and the first order differences. It is therefore desirable to see what kind of higher order discrete Wirtinger inequalities can be constructed to obtain information about the geometry of a polygon, especially those involving more delicate geometric quantities such as the curvature. To this end, we are going to define the curvature of a polygon to be the second order differences, and at the same time construct higher order Wirtinger inequalities which naturally contain the information about the curvature.

Indeed, we obtain a family of inequalities ($I_m$) indexed by $m\in \mathbb N$, which involves up to the $m$-th derivative.
The inequality $I_m$ in a sense measures the regularity of the $k$-gon up to the $m$-th order.
There are two remarkable properties of these inequalities: 1) when $m$ is odd, it provides a lower bound for the isoperimetric deficit, and when $m$ is even, it provides an upper bound; 2) the inequality $I_{m+1}$ measures the stability of $I_m$ by giving an upper bound on its deficit.

Our approach also has the advantage that all the constants that appear are explicit and sharp. See e.g. \cite{indrei2016sharp}, \cite{indrei2015stability} for another approach based on the spectral theory for circulant matrices, with a constant which is less explicit.

We present here the simplified versions of the general inequality $I_m$ stated in Theorem \ref{thm discrete higher} for $m=1$ and $m=2$. When $m=1$, the inequality $I_1$ (Theorem \ref{thm chakerian 2}) is the discrete version of Chakerian's sharpened isoperimetric inequality \cite{Chakerian1978}:
\begin{align*}
2 \cos^{2}\left(\frac{\pi}{k}\right)\left(S(P)-4 \tan \left(\frac{\pi}{k}\right) F(P)\right)\ge\left\|\boldsymbol{t}-2 i \sin \left(\frac{\pi}{k}\right) e^{i \frac{\pi}{k}} z\right\|^{2}.
\end{align*}
The equality holds if and only if $P$ is a positively oriented regular $k$-gon.
Here $S(P)$ is the sum of the squares of the side lengths of the $k$-sided polygon $P$ and $F(P)$ is the algebraic area enclosed by $P$. The vector $z$ denotes the position of the vertices and $\boldsymbol t$ denotes the ``tangent'' vectors at the vertices.

When $m=2$, the inequality $I_2$ in Theorem \ref{thm discrete higher} becomes
\begin{align*}
& 8\left(\sin^{2} \left(\frac{2 \pi}{k}\right)-\sin^{2} \left(\frac{\pi}{k}\right)\right) \cos^{2}\left(\frac{\pi}{k}\right)\left(S(P)-4 \tan \left(\frac{\pi}{k}\right) F(P)\right)\\
\le& 4\left(\sin^{2}\left(\frac{2 \pi}{k}\right)-\sin^{2}\left(\frac{\pi}{k}\right)\right) \left\|\boldsymbol{t}-2 i \sin \left(\frac{\pi}{k}\right) e^{i \frac{\pi}{k}} z\right\|^{2}+\left\| \boldsymbol{\kappa}+4 \sin^{2}\left(\frac{\pi}{k}\right) z \right\|^{2}.
\end{align*}
The terms $\left\|\boldsymbol{t}-2 i \sin \left(\frac{\pi}{k}\right) e^{i \frac{\pi}{k}} z\right\|^2$ and $\left\|\boldsymbol{\kappa}+4 \sin^2\left(\frac{\pi}{k}\right) z\right\|^2$ are natural as they vanish on a positively oriented regular $k$-gon, and in a sense they measure the regularity of the $k$-gon up to second order.
The inequality is sharp and we will characterize the rigidity case.
Here $\boldsymbol \kappa$ is the curvature vector, which is essentially a second order difference of the position vector.

Let us remark that this is the discrete analogue of the following stability inequality from \cite{kwong2021higher}:
$$
\frac{L}{2 \pi^{2}}\left(L^{2}-4 \pi F\right)
\le\int_{C}\left|z-\left(\frac{L}{2 \pi}\right) \boldsymbol n \right|^{2} d s+\frac{1}{3} \int_{C}\left|z+\left(\frac{L}{2 \pi}\right)^{2} \boldsymbol\kappa\right|^{2} d s.
$$
It can also be rearranged so that it gives a measure of the stability of the discrete Chakerian's isoperimetric inequality, see Remark \ref{rmk2}.

\begin{equation*}
\begin{split}
\end{split}
\end{equation*}
In the second part of our paper, starting from Section \ref{sec smooth}, we employ the smooth version of the higher-order Wirtinger inequalities to establish the smooth counterpart of the isoperimetric type inequalities. Our main result is Theorem \ref{smooth thm}, which in a sense quantifies the ``roundness'' of the curve up to order $m$.
\begin{theorem}[Theorem \ref{smooth thm}]
Let $C$ be a simple closed $C^{m}$ curve in $\mathbb C$ with the length $L$ and a unit speed counter-clockwise parametrization $z (s) $ with $s \in[0, L] $.
Assume that $C =\partial \Omega$ is the boundary curve of a domain $\Omega \subset \mathbb C$ with the area $F$. Let $\boldsymbol{n} $ denotes the outward-pointing unit normal on $C$. Assume $C$ has centroid $0$, then
\begin{equation*}
\begin{split}
0 \le & \sum_{l=1}^{m-1} s_{m, l}\left(\frac{L}{2 \pi}\right)^{2 l-3} \int_{C}\left|\left(\frac{d}{d s}\right)^{l-1}\left(z+\left(\frac{L}{2 \pi}\right)^{2} \boldsymbol{\kappa}\right)\right|^{2} d s \\
& -\frac{(-1)^{m}}{2}(m-1) !(m+1) !\left(\frac{1}{\pi}\left(L^{2}-4 \pi F\right)-\frac{2 \pi}{L} \int_{C}\left|z-\left(\frac{L}{2 \pi}\right) \boldsymbol{n}\right|^{2} d s\right).
\end{split}
\end{equation*}
The equality holds if and only if $z(t)=be^{it}$ for some $b\in \mathbb C$.
Here, the constants $s_{m, l}$ are explicit.
\end{theorem}
From a Fourier perspective, the integral $\sum_{l=1}^{m-1} s_{m, l}\left(\frac{L}{2 \pi}\right)^{2 l-1} \int_C\left|\left(\frac{d}{d s}\right)^{l-1}\left(z+\left(\frac{L}{2 \pi}\right)^2 \kappa\right)\right|^2 d s$ is a natural quantity to consider. It is worth noting that
a recent paper \cite{mccoy2023representation} introduced similar quantities, referred to as $m$-th energy functionals. This is not surprising, as their energy functionals are also based on a Fourier approach. However, unlike the energy functionals discussed here, the functionals introduced in \cite{mccoy2023representation} are limited to convex curves.

Finally, we will generalize the Chernoff area-width inequality. The classical Chernoff inequality states that for a closed convex curve $\gamma$ on $\mathbb R^{2}$ with area $F$ and width function $w(\theta)$,
$$
F \le \frac{1}{2} \int_{0}^{\frac{\pi}{2}} w(\theta) w\left(\theta+\frac{\pi}{2} \right) d \theta.
$$
Ou and Pan \cite{ou2010some} generalized the Chernoff inequality by considering the so called $k$-width of the curve, instead of the ordinary width:
$$
F \le \frac{1}{k} \int_{0}^{\frac{\pi}{k}} w_{k}(\theta) w_{k}\left(\theta+\frac{\pi}{k}\right) d \theta.
$$
This inequality reduces to the classical Chernoff inequality when $k=2$, and recovers the classical isoperimetric inequality when $k$ approaches infinity. As we have obtained higher order isoperimetric inequalities using Fourier analysis, it is natural to ask whether we can extend this method to obtain higher order versions of Ou-Pan type inequalities. In Section \ref{chernoff}, we address this question and introduce a family $(C_m)$ of $m$-th order generalized Chernoff inequalities. For further details, see Theorem \ref{Chernoff}.

To keep this introduction brief, we will leave the outline of the idea of the discrete or smooth Wirtinger inequalities to the next section, but let us just emphasize that the same approach is used to obtain the smooth isoperimetric type inequalities (Section \ref{sec smooth}) for plane curves, higher order Poincare type inequalities for hypersurfaces in Euclidean space of any dimension (which we are not going to cover in this paper, see \cite{kwong2021higher2}), and generalized Chernoff inequalities (Section \ref{chernoff}) that involve a generalized width, higher order locus of curvature centers and the mixed area.

The organization of this paper is as follows. In Section \ref{sec idea}, we outline the idea to obtain higher order Wirtinger inequalities on a smooth manifold, which can be easily adapted to the discrete case. In Section \ref{sec background}, we provide the basic knowledge of discrete Fourier analysis, which will then be used to prove discrete Wirtinger inequalities in Section \ref{sec discrete wirtinger}. These inequalities are in turn used to prove geometric discrete isoperimetric type inequalities in Section \ref{sec geom}. We then switch to the smooth case in Section \ref{sec smooth}, in which we prove the smooth analogues of the upper and lower bounds of the isoperimetric deficits. Section \ref{sec smooth} is independent of Sections \ref{sec background}, \ref{sec discrete wirtinger} and \ref{sec geom}, although the underlying idea is the same. Finally, in Section \ref{chernoff}, we prove a generalization of the Chernoff inequality.

\textbf{Acknowledgements}: We would like to thank Emanuel Indrei for suggesting this problem and for useful comments. The author is partially supported by the CERL fellowship at University of Wollongong and by grant FL150100126 of the Australian Research Council.

\section{ Basic idea to generate higher order inequalities }\label{sec idea}
The idea to generate higher order Wirtinger or Poincare type inequalities is as follows. For the sake of discussion, we will discuss the case of a Riemannian manifold, although the case for a polygon is completely analogous, and let us ignore any convergence or regularity issue in this section, and more details will be given later. (Of course, in the discrete case, there is no such issue.)
Let $M$ be a closed Riemannian manifold and $\{T_j\}_{j=1}^{m}: C^{\infty}(M)\to C^{\infty}(M)$ be a family of self-adjoint differential operators which mutually commute with each other. Then there exists an orthonormal Schauder basis $\{e_n\}$ which are simultaneous eigenfunctions of the $T_j$, with eigenvalues $\{\lambda_{n, j}\}$, i.e. $T_j [e_n] =\lambda_{n, j}e_n$.
\begin{example}
Let us give some examples of such operators. If $M=\mathbb S^{1}$ is the standard unit circle, examples of such operators include differential operators of the form $ \displaystyle T[h]=\sum_{j=0}^{m}c_j h^{(2j)}$. An important case is $T[h]=h+\ddot h$, which gives the reciprocal of the curvature of a convex curve $\gamma$ if $h=h(\theta)$ is the support function of a convex curve, parametrised by the normal angle $\theta$ (\cite[p. 34]{ChouZhu2001}).
Another class of such operators is the averaged translational operator such as
$\displaystyle T_k[h](\theta)=\frac{1}{k} \sum_{j=1}^{k} h\left(\theta+\frac{(2 j-1)\pi}{k}\right)$ for some $\displaystyle k\in \mathbb N$. This operator is the generalization of the Chernoff operator $T_2$.
In 1969, Chernoff \cite{chernoff1969area} obtained an area-width inequality for convex plane curves: Let $\gamma$ be a closed convex curve in the plane $\mathbb R^{2}$ with area $F$ and width function $w(\theta)$, then
$$
F \le \frac{1}{2} \int_{0}^{\pi / 2} w(\theta) w\left(\theta+\frac{1}{2} \pi\right) d \theta,
$$
and the equality holds if and only if $\gamma$ is a circle. Indeed, the Chernoff inequality can be regarded as a Wirtinger inequality for $T_2$. By obtaining a first order Wirtinger-type inequality for $T_k$, Ou and Pan \cite{ou2010some} obtained a generalized version of the Chernoff inequality:
$$
F \le \frac{1}{k} \int_{0}^{\pi / k} w_{k}(\theta) w_{k}\left(\theta+\frac{1}{k} \pi\right) d \theta,
$$
where $w_{k}(\theta)=h(\theta)+h\left(\theta+\frac{2 \pi}{k}\right)+\cdots+h\left(\theta+\left(\frac{2(k-1) \pi}{k}\right)\right)$ and $h$ is the support function. Note that Ou and Pan did not explicitly state the operator $T_k$ in their paper, but their analysis essentially focused on the spectral properties of $T_k$.

If $M=\mathbb S^{d-1}$, we can consider the family of operators $\{-\Delta -\lambda_j\mathrm{Id}\}_{j=1}^{m}$, where $\Delta$ is the (negative definite) Laplacian and
$0=\lambda_{0}<\lambda_{1}<\lambda_{2}<\cdots \rightarrow \infty$ are the eigenvalues of $-\Delta$ on $\mathbb S^{d-1}$. In \cite{kwong2021higher2}, the author considered this family of operators and obtained higher order Poincare inequalities on the sphere $\mathbb S^{d-1} $. As a result, new Minkowski-type inequalities involving higher order mean curvatures for convex hypersurfaces in $\mathbb R^d$ were also obtained.
\end{example}

Let us go back to the recipe to obtain higher order Wirtinger type inequalities. Suppose we can construct a family of such $\{T_j\}_{j=1}^{m}$ such that $\prod_{j=1}^{m} \lambda_{n, j}\ge0$ for all $n$. Suppose $f\in C^{\infty}(M)$ can be expressed in the orthonormal Schauder basis as $f=\sum_{n} a_{n} e_n$, then
\begin{align*}
\int_M f\cdot\prod_{j=1}^{m} T_j [f]
=\sum_{n}\left(\prod_{j=1}^{m} \lambda_{n, j} \right)|a_n|^2
\ge 0.
\end{align*}
This is our higher order Wirtinger inequality. Although the idea seems simple, when we suitably choose $f$ and $T_j$, and with some rearrangement, this gives us some geometric inequalities and stability results that are new.
\section{Background knowledge: discrete case}\label{sec background}
We will start from the discrete case first.
Let us first give the necessary definitions and notation. Let $k\in \mathbb N$, which we are going to fix.
A $k$-gon $P$ in the Euclidean plane is an ordered $k$-tuple of points, $P =\left(z_{0}, z_{1}, \ldots, z_{k-1}\right). $ These points are called vertices, and the line segments joining $z_{0}$ to $z_{1}, z_{1}$ to $z_{2}, \ldots, z_{k-1}$ to $z_{0}$ are called sides.
We identify a point on the plane with a complex number. For convenience, we also identify $z_{k-l}$ with $z_{-l}$ (and later on $\zeta_{k-l}$ with $\zeta_{-l}$), so for example $\{z_\nu: |\nu|\le m\}=\{z_\nu: 0\le \nu\le m \textrm{ or }k-m\le \nu\le k-1\}$.

We can then define scaling by complex numbers and linear combination of two polygons. So $\left(r e^{i \theta}\right) P$ is the image of $P$ under a dilatation by the real factor $r$ and a rotation through the angle $\theta$, both about the origin.
A linear combination of $k$-gons $P$ and $Q$ is the $k$-gon identified with the linear combination of the corresponding vectors in $\mathbb C^{k}$.

Let $\tau^{n} P$ be the $k$-gon obtained from $P$ by a cyclic shift of the entries of
$\left(z_{0}, z_{1}, \ldots, z_{k-1}\right) n$ places to the right. That is,
\begin{equation*}
\begin{split}
\tau P =\left(z_{k-1}, z_0, z_{1}, z_{2}, \ldots, z_{k-2} \right), \;
\tau^{2} P =\left(z_{k-2}, z_{k-1}, \ldots, z_{k-3}\right), \;\text { etc. }
\end{split}
\end{equation*}

We will denote by $R_{n}$ the special $n$-regular $k$-gon centered at 0 with $z_{0}=1$, namely,
$$
R_{n}=\left(1, \omega^{n}, \omega^{2 n}, \ldots, \omega^{(k-1) n}\right)
$$
where $\omega=e^{\frac{2 \pi i}{k} } $ is a $k$-th root of unity.

Since our results will be proved using finite Fourier series, it is more convenient to work with various Hermitian forms on the vector space $\mathbb C^{k}$ of all $k$-gons. In particular, we say a Hermitian form $I$ which is invariant under $\tau$ (i.e. $I(\tau P, \tau Q)=I(P, Q)$) a polygonal form (\cite[Section 3]{fisher1985perpendicular}). For simplicity we denote a Hermitian form and the quadratic form associated to it by the same symbol.

The signed area of $P$ is given by the polygonal form \cite[p. 30]{fisher1985perpendicular}
\begin{align}\label{F}
F(P):=\frac{1}{4 i} \sum_{j=0}^{k-1}\left(z_{j+1} \bar{z}_{j}-z_{j} \bar{z}_{j+1}\right).
\end{align}

Another polygonal form that we will consider here is $S=S(P)$, the sum of the squares of the side lengths of the polygon $P$.
Thus from the definition
\begin{equation}\label{S}
S(P)=\sum_{j=0}^{k-1}\left|z_{j+1}-z_{j}\right|^{2}=\sum_{j=0}^{k-1}\left(z_{j+1}-z_{j}\right)\left(\bar{z}_{j+1}-\bar{z}_{j}\right).
\end{equation}
It turns out that in the Fourier series approach, $S$ is a better quantity to control than the length $L=\sum_{\nu=0}^{k-1}|z_{\nu+1}-z_\nu|$ of the polygon. This is because $S$ is a polygonal form whereas the squared length $L^2$ is not even a Hermitian form. This is analogous to the fact that Fourier series is harder to be used to control the $L^p$ norm of a function or its derivatives for $p$ other than $2$. In many instances, the isoperimetric deficit $L^2-4\pi F$ in the smooth case is replaced by the deficit $S-4 \tan \left(\frac{\pi}{k}\right) F$ in the discrete case, see for example Theorem \ref{thm1}.

For the convenience of the reader, we now provide some basic knowledge on finite Fourier series here. For more details on finite Fourier series, please refer to \cite{schoenberg1950finite}.

Let
$z= \left(z_0, \cdots z_{k-1} \right)\in \mathbb C^k$.
The finite Fourier series (FFS) of $z$ is
$$
z_{\nu}=\zeta_{0}+\zeta_{1} \omega_{\nu}+\zeta_{2} \omega_\nu^{2}+\cdots+\zeta_{k-1} \omega_{\nu}^{k-1} \quad(\nu=0, 1, \cdots, k-1)
$$
where $\omega_\nu=e^{\frac{2 \pi i \nu }{k}}$. It is well-known that the Fourier coefficients $\zeta_\nu$ are given by
$\zeta_\nu=\frac{1}{k} \sum_{j=0}^{k-1}z_j \bar \omega_\nu^j =\frac{1}{k}\langle z, R_\nu\rangle$,
so that $z=\sum_{\nu=0}^{k-1}\zeta_\nu R_\nu$.

We define $\dot{z}$ by
\begin{equation*}
\dot{z}=
\left(z_1-z_0, z_{2}-z_{1}, \cdots, z_0-z_{k-1}\right).
\end{equation*}
Up to a sign, this is called the first difference, nevertheless we prefer to call it the ``first derivative'' in order to compare with \cite{kwong2021higher}.

Using this notation,
\eqref{S} and \eqref{F} become, respectively,
\begin{equation*}
S(P)=\|\dot P\|^2=\langle \dot P, \dot P\rangle
\quad \textrm{and }\quad
F(P) =-\frac{1}{2} \mathrm{Im}\langle P, \dot P\rangle,
\end{equation*}
where the $\langle\cdot, \cdot\rangle$ is the standard Hermitian inner product given by $\langle z, w\rangle=\sum_{\nu=0}^{k-1} z_{\nu} \overline{w_{\nu}}$.
This is the discrete analogue of the area enclosed by a smooth curve $C$:
$2 \mathrm{Area}
=\int_C xdy-ydx
=-\int_{C}\mathrm{Im}(z \dot{\bar{z}}) d t$.

In terms of FFS, we have
$\dot {z}
=\sum_{\nu=0}^{k-1} \zeta_{\nu} \dot R_{\nu}
=\sum_{\nu=0}^{k-1} \zeta_{\nu} (\omega_\nu-1)R_{\nu}
=\sum_{\nu=1}^{k-1} \zeta_{\nu} (\omega_\nu-1)R_{\nu}$.
So we have
$\left\|\dot{z}\right\|^{2}=k \sum_{\nu=1}^{k-1}\left|\zeta_{\nu}\right|^{2}\left|1-\omega_{\nu}\right|^{2}$.
Naturally we define the ``higher derivatives'', by defining $\ddot z$ to be the derivative of $\dot z$, etc. In the literature, these are called the differences of higher order. Denote the $j$-th derivative of $z$ by either $z^{(j)}$ or $D^jz$.
Inductively, we have the Parseval identity
\begin{align}\label{parseval}
\frac{1}{k}\left\|z^{(j)}\right\|^{2}
=\sum_{\nu=0}^{k-1}\left|1-\omega_{\nu}\right|^{2 j}\left|\zeta_{\nu}\right|^{2}
=\sum_{\nu=0}^{k-1}4^j\sin^{2j}\left(\frac{\nu\pi}{k}\right)\left|\zeta_{\nu}\right|^{2}.
\end{align}

We will also make frequent use of the fact that for a polygonal form $I$ (\cite[Theorem 3.6]{fisher1985perpendicular})
\begin{align*}
I(z)=\sum_{\nu=0}^{k-1}|\zeta_\nu|^2 I(R_\nu).
\end{align*}

\section{Discrete Wirtinger inequalities}\label{sec discrete wirtinger}

We now give a family of sharp discrete Wirtinger inequalities, which involves any arbitrary order of derivatives.
\begin{proposition}\label{prop1}
Let $m \le \left\lfloor\frac{k}{2}\right\rfloor$ be a natural number.
Define $m$-th degree polynomial $Q_m(x)=\prod_{j=1}^{m}\left(x-4 \sin^{2}\left(\frac{j \pi}{k}\right)\right)=\sum_{j=0}^{m} c_{m, j}x^j$.
Suppose $z=(z_0, \cdots, z_{k-1})$ satisfies $\sum_{\nu=0}^{k-1}z_\nu=0$, then
\begin{equation}\label{ineq c}
\begin{split}
\sum_{j=0}^{m} c_{m, j}\left\|z^{(j)}\right\|^{2} \ge 0.
\end{split}
\end{equation}
The equality holds if and only if the Fourier coefficients satisfy $\zeta_\nu=0$ for $m<\nu<k-m$.
\end{proposition}

\begin{proof}
As $\sum_{\nu=0}^{k-1} z_{\nu}=0$, we have $\zeta_{0}=0$, where $z=\sum_{\nu=0}^{k-1} \zeta_{\nu} R_{\nu}$ is its finite Fourier series.
By \eqref{parseval},
\begin{equation*}
\begin{split}
\frac{1}{k} \sum_{j=0}^{m} c_{m, j}\left\|z^{(j)}\right\|^{2} & =\sum_{\nu=1}^{k-1}\left[\sum_{j=0}^{m} c_{m, j}\left|1-\omega_{\nu}\right|^{2 j}\right]\left|\zeta_{\nu}\right|^{2} \\
& =\sum_{\nu=1}^{k-1} Q_{m}\left(\left|1-\omega_{\nu}\right|^{2}\right)\left|\zeta_{\nu}\right|^{2} \\
& =\sum_{\nu=1}^{k-1}\left[\prod_{j=1}^{m}\left(\left|1-\omega_{\nu}\right|^{2}-4 \sin^{2}\left(\frac{j \pi}{k}\right)\right)\right]\left|\zeta_{\nu}\right|^{2} \\
& =\sum_{\nu=1}^{k-1} \left[\prod_{j=1}^{m}\left(\left|1-\omega_{\nu}\right|^{2}-\left|1-\omega_{j}\right|^{2}\right)\right]\left|\zeta_{\nu}\right|^{2} \\
& =\sum_{\nu=m+1}^{k-m-1}\left[\prod_{j=1}^{m}\left(\left|1-\omega_{\nu}\right|^{2}-\left|1-\omega_{j}\right|^{2}\right)\right]\left|\zeta_{\nu}\right|^{2} \\
& \ge 0.
\end{split}
\end{equation*}

Here, we have used the fact that
\begin{equation*}
\begin{split}
\left|1-\omega_{\nu}\right|^{2} =\left(\cos \left(\frac{2 \pi \nu}{k}\right)-1\right)^{2}+\sin^{2}\left(\frac{2 \pi \nu}{k}\right)
=2-2 \cos \left(\frac{2 \pi \nu}{k}\right)
=4 \sin^{2}\left(\frac{\nu \pi}{k}\right).
\end{split}
\end{equation*}

Hence for $m \le \left\lfloor\frac{k}{2} \right\rfloor$,
$$|1-\omega_1|^2=|1-\omega_{k-1}|^2< |1-\omega_2|^2=|1-\omega_{k-2}|^2<\cdots< |1-\omega_m|^2=|1-\omega_{k-m}|^2$$
and all $|1-\omega_\nu|^2> |1-\omega_m|^2$ for $m<\nu<k-m$.
\end{proof}

To obtain geometric applications of the inequality \eqref{ineq c}, we rearrange it so that each term has its geometric meaning. We are going to do it in two steps. The form that we are looking for is the inequality \eqref{ineq s} and we are going to give its geometric interpretation in Section \ref{sec geom}.
\begin{proposition}\label{prop lambda}
Let $m \le \left\lfloor \frac{k}{2}\right\rfloor$ be a natural number.
Define $P_1=1$ and the $(m-1)$-th degree polynomial $\displaystyle P_{m}(x):=\prod_{j=2}^{m}\left(x-4 \sin^{2}\left(\frac{j \pi}{k}\right)\right)=\sum_{k=0}^{m-1} \lambda_{m, k} x^{k}$ for $m\ge 2$.
Suppose $\displaystyle z=\left(z_{0}, \cdots, z_{k-1}\right)$ satisfies $\displaystyle \sum_{\nu=0}^{k-1} z_{\nu}=0$, then
$$\sum_{j=0}^{m-1} \lambda_{m, j}\left(\left\|z^{(j+1)}\right\|^{2}-4 \sin^{2}\left(\frac{\pi}{k}\right)\left\|z^{(j)}\right\|^{2} \right)\ge0. $$
\end{proposition}

\begin{proof}
From the relation $Q_{m}(x)=\left(x-4 \sin^{2}\left(\frac{\pi}{k}\right)\right) P_{m}(x)$, we can deduce that
$ c_{m, j}=\lambda_{m, j-1}-4 \sin^{2}\left(\frac{\pi}{k}\right) \lambda_{m, j}$, where we use the standard convention that $\lambda_{m, k}=0$ for $k \notin\{0, \cdots, m-1\}$.
From this and Proposition \ref{prop1} we have
\begin{equation*}
\begin{split}
0 \le \sum_{j=0}^{m} c_{m, j}\left\|z^{(j)}\right\|^{2}=& \sum_{j=0}^{m}\left(\lambda_{m, j-1}-4 \sin^{2}\left(\frac{\pi}{k}\right) \lambda_{m, j}\right)\left\|z^{(j)}\right\|^{2} \\
=& \sum_{j=0}^{m-1} \lambda_{m, j}\left(\left\|z^{(j+1)}\right\|^{2}-4 \sin^{2}\left(\frac{\pi}{k}\right)\left\|z^{(j)}\right\|^{2} \right).
\end{split}
\end{equation*}
\end{proof}

\begin{proposition}\label{prop s}
Let $m \le \left\lfloor\frac{k}{2}\right\rfloor$ be a natural number.
Define the $(m-2)$-th degree polynomial $\displaystyle \mathcal S_{m}(x)=\frac{P_{m}(x)-P_{m}\left(4 \sin^{2} \left(\frac{\pi}{k}\right)\right)}{x-4 \sin^{2}\left(\frac{\pi}{k}\right)}=\sum_{l=1}^{m-1} S_{m, l} x^{l-1} $ and $S_{m, 0}=P_{m}\left(4 \sin^{2}\left(\frac{\pi}{k}\right)\right)$.
Suppose $z=\left(z_{0}, \cdots, z_{k-1}\right)$ satisfies $\sum_{\nu=0}^{k-1} z_{\nu}=0$, then
\begin{equation}\label{ineq s}
\begin{split}
0 & \le S_{m, 0}\left(\|\dot{z}\|^{2}-4 \sin^{2}\left(\frac{\pi}{k}\right)\|z\|^{2}\right)+\sum_{l=1}^{m-1} S_{m, l}\left\|\tau z^{(l+1)}+4 \sin^{2}\left(\frac{\pi}{k}\right) z^{(l-1)}\right\|^{2}.
\end{split}
\end{equation}
The equality holds if and only if the Fourier coefficients of $z$ satisfy $\zeta_{\nu}=0$ for $m<\nu<k-m$.
\end{proposition}

\begin{proof}
By definition,
\begin{equation*}
\begin{split}
P_{m}(x) & =\left(x-4 \sin^{2} \frac{\pi}{k}\right) \mathcal S_{m}(x)+P_{m}\left(4\sin^{2} \left(\frac{\pi}{k}\right)\right) \\
& =\left(x-4 \sin^{2}\left(\frac{\pi}{k}\right)\right) \mathcal S_{m}(x)+S_{m, 0} \\
& =\sum_{l=0}^{m-1}\left(S_{m, l}- 4 \sin^{2}\left(\frac{\pi}{k}\right) S_{m, l+1}\right) x^{l}.
\end{split}
\end{equation*}
Therefore
\begin{equation}\label{lambda}
\lambda_{m, l}=S_{m, l}-4 \sin^{2}\left(\frac{\pi}{k}\right) S_{m, l+1}.
\end{equation}

Define $I_{l}=\left\|\tau z^{(l+1)}+4 \sin^{2}\left(\frac{\pi}{k}\right)z^{(l-1)}\right\|^{2}$ and $J_{l}=\left\|z^{(l+1)}\right\|^{2}-4\sin^2\left(\frac{\pi}{k}\right)\left\|z^{(l)}\right\|^{2}$. We have
\begin{equation}\label{Il}
\begin{split}
I_{l} & =\left\|\tau z^{(l+1)}+4 \sin^{2}\left(\frac{\pi}{k}\right) z^{(l-1)}\right\|^{2} \\
& =\left\|z^{(l+1)}\right\|^{2}+16 \sin^{4}\left(\frac{\pi}{k}\right)\left\|z^{(l-1)}\right\|^{2}+8 \sin^{2}\left(\frac{\pi}{k}\right)\left\langle\tau z^{(l+1)}, z^{(l-1)}\right\rangle \\
& =\left\|z^{(l+1)}\right\|^{2}+16 \sin^{4}\left(\frac{\pi}{k}\right)\left\|z^{(l-1)}\right\|^{2}-8 \sin^{2}\left(\frac{\pi}{k}\right)\left\|z^{(l)}\right\|^{2} \\
& =\left(\left\|z^{(l+1)}\right\|^{2}-4 \sin^{2}\left(\frac{\pi}{k}\right)\left\|z^{(l)}\right\|^{2}\right)-4 \sin^{2}\left(\frac{\pi}{k}\right)\left(\left\|z^{(l)}\right\|^{2}-4\sin^{2}\left(\frac{\pi}{k}\right)\left\|z^{(l-1)}\right\|^{2}\right) \\
& =J_{l}-4 \sin^{2}\left(\frac{\pi}{k}\right) J_{l-1}.
\end{split}
\end{equation}
Here we have used the summation by parts formula $\langle\dot{w}, \dot{w}\rangle=-\left\langle\ddot{w}, \tau^{-1} w\right\rangle=-\langle\tau \ddot{w}, w\rangle$.
So by Proposition \ref{lambda}, \eqref{lambda} and \eqref{Il}, we have
\begin{equation*}
\begin{split}
0 \le& \sum_{l=0}^{m-1} \lambda_{m, l} J_{l}=\sum_{l=0}^{m-1}\left(S_{m, l}-4 \sin^{2}\left(\frac{\pi}{k}\right) S_{m, l+1}\right) J_{l}
=S_{m, 0} J_{0}+\sum_{l=1}^{m-1} S_{m, l}\left(J_{l}-4 \sin^{2}\left(\frac{\pi}{k}\right) J_{l-1}\right)\\
=& S_{m, 0} \left(\left\|\dot z\right\|^{2}-4 \sin^{2}\left(\frac{\pi}{k}\right)\left\|z\right\|^{2}\right)+\sum_{l=1}^{m-1} S_{m, l}\left\|\tau z^{(l+1)}+4 \sin^{2}\left(\frac{\pi}{k}\right) z^{(l-1)}\right\|^{2}.
\end{split}
\end{equation*}
\end{proof}
Using the recurrence relations satisfied by the polynomials $Q_m$ and $S_{m}$, we can also derive stability results for the inequalities in Proposition \ref{prop1} and Proposition \ref{prop s}.
\begin{proposition}\label{prop stab}
Let $m \le \left\lfloor\frac{k}{2}\right\rfloor-1$ be a natural number and
$z=\left(z_0, \cdots, z_{k-1}\right)$ with $\sum_{\nu=0}^{k-1} z_\nu=0$.
With the notation of Proposition \ref{prop1} and \ref{prop s}, the following inequalities hold.
\begin{enumerate}
\item
\begin{equation}\label{stab1}
\sum_{l=0}^{m} C_{m, l}\left\|z^{(l)}\right\|^{2}
\le \frac{1}{4 \sin^{2}\left(\frac{m+1) \pi}{k}\right)}\sum_{l=0}^{m} C_{m, l}\left\|z^{(l+1)}\right\|^{2}.
\end{equation}
\item
\begin{equation}\label{stab2}
\begin{split}
& S_{m, 0}\left(\|\dot{z}\|^{2}-4 \sin^{2}\left(\frac{\pi}{k}\right)\|z\|^{2}\right)+\sum_{l=1}^{m-1} S_{m, l}\left\|\tau z^{(l+1)}+4 \sin^{2}\left(\frac{\pi}{k}\right) z^{(l-1)}\right\|^{2}\\
& \le \frac{1}{4 \sin^{2}\left(\frac{(m+1) \pi}{k}\right)}\left[4\sin^{2}\left(\frac{\pi}{k}\right)S_{m, 0} \left(\|\dot{z}\|^{2}-4 \sin^{2}\left(\frac{\pi}{k}\right)\|z\|^{2}\right)+\sum_{l=0}^{m-1} S_{m, l} \|\tau z^{(l+2)}+4 \sin^{2}\left(\frac{\pi}{k}\right)z^{(l)}\|^{2} \right].
\end{split}
\end{equation}
\end{enumerate}
The equality for both inequalities hold if and only if $z=\sum_{0<|\nu|\le m+1} \zeta_\nu R_\nu$.
\end{proposition}

\begin{proof}
By the identity $Q_{m+1}(x)=\left(x-4 \sin^{2}\left(\frac{(m+1) \pi}{k}\right)\right) Q_{m}(x)$, we have the recurrence relation
\begin{equation*}
C_{m+1, l}=C_{m, l-1}-4 \sin^{2}\left(\frac{m+1) \pi}{k}\right) C_{m, l}.
\end{equation*}
Here, we use the standard convention that $c_{m, k}=0$ for $k \notin\{0, \cdots, m\} $. By replacing $m$ with $m+1$, the inequality \eqref{ineq c} then gives
\begin{equation*}
\begin{split}
0 & \le \sum_{l=0}^{m+1} C_{m+1, l}\left\|z^{(l)}\right\|^{2} \\
& =\sum_{l=0}^{m+1}\left[C_{m, l-1}-4 \sin^{2}\left(\frac{m+1) \pi}{k}\right) C_{m, l}\right]\left\|z^{(l)}\right\|^{2} \\
& =\sum_{l=0}^{m} C_{m, l}\left\|z^{(l+1)}\right\|^{2}-4 \sin^{2}\left(\frac{(m+1) \pi}{k}\right) \sum_{l=0}^{m} C_{m, l}\left\|z^{(l)}\right\|^{2}.
\end{split}
\end{equation*}
From this \eqref{stab1} follows.

The proof for \eqref{stab2} is similar. It requires the recurrence relation $S_{m+1, l}=S_{m, l-1}-4 \sin^{2}\left(\frac{m+1) \pi}{k}\right) S_{m, l}$, for $l \ge 1$, which is obtained from the identity $\mathcal S_{m+1}(x)=\left(x-4 \sin^{2}\left(\frac{m+1) \pi}{k}\right)\right) \mathcal S_{m}(x)+S_{m, 0}$.
\end{proof}

\section{Geometric inequalities}\label{sec geom}
Recall that $S=S(P)$ and $F=F(P)$ are the sum of squared sides and the signed area of the polygon $P$ respectively.
If $P$ is a $k$-gon represented by $z\in \mathbb C^k$, then
$$F(P)=-\frac{1}{2} \mathrm{Im}\left\langle z, \dot{z}\right\rangle\quad \textrm{and}\quad S(P)=\|\dot{z}\|^2. $$

For a polygon $z$, we define $\boldsymbol{t}=\dot{z}$ to be its set of tangent vectors (at the corresponding vertices $z_\nu$) and
$ \boldsymbol{\kappa} =\tau\ddot z$ to be the discrete curvature vectors (again at the vertices $z_\nu$). Note that the translation $\tau$ is necessary because
$\tau\ddot z$ at the $\nu$-th position corresponds to the change of the tangent vectors across the vertex $z_\nu$, i.e. it is the change from $\overrightarrow{z_{\nu-1}z_\nu}$ to $\overrightarrow {z_\nu z_{\nu+1}}$. Also, it is not true that $\boldsymbol{\kappa}$ is perpendicular to $\boldsymbol{t}$ in general, unlike the smooth case.

For a regular polygon with centroid $0$ which is positively oriented, $\boldsymbol{t}=\dot z= 2i e^{i \frac{\pi}{k}} \sin \left(\frac{\pi}{k}\right) z $. Therefore the quantity $\left\|\dot{z}-2 i \sin \left(\frac{\pi}{k}\right) e^{i \frac{\pi}{k}} z\right\|^{2}$ measures how far a given polygon $P$ deviates from a regular $k$-gon.

\begin{proposition}\label{prop4}
Let $P$ be a $k$-gon represented by $z$. Then
\begin{equation*}
\begin{aligned}
\|\dot z\|^{2}-4 \sin^{2}\left(\frac{\pi}{k}\right)\|z\|^{2}
=2 \cos^{2}\left(\frac{\pi}{k}\right)\left(S(P)-4\tan \left(\frac{\pi}{k}\right) F(P)\right)-\left\|\dot{z}-2 i \sin \left(\frac{\pi}{k}\right) e^{i \frac{\pi}{k}} z\right\|^{2}.
\end{aligned}
\end{equation*}
\end{proposition}
\begin{proof}
By writing $z=\sum_{\nu=0}^{k-1} \zeta_{\nu} R_{\nu}$, we have
\begin{equation}\label{pf1}
\begin{split}
& \left\|\dot{z}-2 i \sin \left(\frac{\pi}{k}\right) e^{i \frac{\pi}{k}} z\right\|^{2}+\|\dot{z}\|^{2}-4 \sin^{2}\left(\frac{\pi}{k}\right)\|z\|^{2} \\
=& 2\|\dot{z}\|^{2}-8 \sin \left(\frac{\pi}{k}\right)\left[-\frac{1}{2} \mathrm{Im}\left\langle e^{i \frac{\pi}{k}} z, \dot{z}\right\rangle\right]\\
=& 2\|\dot{z}\|^{2}-8 \sin \left(\frac{\pi}{k}\right)\left[-\frac{1}{2} \sum_{\nu=0}^{k-1}\left|\zeta_{\nu}\right|^{2} \mathrm{Im}\left\langle e^{i \frac{\pi}{k}} R_{\nu}, \dot R_{\nu}\right\rangle\right].
\end{split}
\end{equation}
By direct computation, $-\frac{1}{2} \mathrm{Im}\left\langle e^{i \frac{\pi}{k}} R_{\nu}, \dot{R}_{\nu}\right\rangle=k \sin \left(\frac{\nu \pi}{k}\right) \cos \left(\frac{(\nu-1) \pi}{k}\right)$ and $-\frac{1}{2} \mathrm{Im}\left\langle R_{\nu}, \dot{R}_{\nu}\right\rangle=$ $F\left(R_{\nu}\right)=k \sin \left(\frac{\nu \pi}{k}\right) \cos \left(\frac{\nu \pi}{k}\right)$. Therefore
\begin{equation*}
\begin{split}
& \left\|\dot{z}-2 i \sin \left(\frac{\pi}{k}\right) e^{i \frac{\pi}{k}} z\right\|^{2}+\|\dot{z}\|^{2}-4 \sin^{2}\left(\frac{\pi}{k}\right)\|z\|^{2}\\
=& 2\|\dot{z}\|^{2}-8 \sin \left(\frac{\pi}{k}\right) \sum_{\nu=0}^{k-1}\left|\zeta_{\nu}\right|^{2} k \sin \left(\frac{\nu \pi}{k}\right) \cos \left(\frac{(\nu-1) \pi}{k}\right) \\
=& 2\|\dot{z}\|^{2}-8 \sin \left(\frac{\pi}{k}\right) \sum_{\nu=0}^{k-1}\left|\zeta_{\nu}\right|^{2} k \sin \left(\frac{\nu \pi}{k}\right)\left(\cos \left(\frac{\nu \pi}{k}\right) \cos \left(\frac{\pi}{k}\right)+\sin \left(\frac{\nu \pi}{k}\right) \sin \left(\frac{\pi}{k}\right)\right) \\
=& 2\|\dot{z}\|^{2}-8 \sin \left(\frac{\pi}{k}\right) \cos \left(\frac{\pi}{k}\right) \sum_{\nu=0}^{k-1}\left|\zeta_{\nu}\right|^{2} F\left(R_{\nu}\right)-2 \sin^{2}\left(\frac{\pi}{k}\right) \sum_{\nu=0}^{k-1}\left|\zeta_{\nu}\right|^{2} S\left(R_{\nu}\right) \\
=& 2 S(P)-8 \sin \left(\frac{\pi}{k}\right) \cos \left(\frac{\pi}{k}\right) F(P)-2 \sin^{2}\left(\frac{\pi}{k}\right) S(P) \\
=& 2 \cos^{2}\left(\frac{\pi}{k}\right)\left(S(P)-4 \tan \left(\frac{\pi}{k}\right) F(P)\right),
\end{split}
\end{equation*}
where we have used $S\left(R_{\nu}\right)=4 k \sin^{2}\left(\frac{\nu \pi}{k}\right)$.
\end{proof}

In \cite{Chakerian1978}, Chakerian proved the following sharpened isoperimetric inequality for a simple closed curve with centroid $0$:
$$
\frac{2 \pi^{2}}{L} \int_{ C }\left| z -\left(\frac{L}{2 \pi}\right) \boldsymbol{n} \right|^{2} d s \le L^{2}-4 \pi F.
$$
A discrete analogue of Chakerian's result is given by the following result. Indeed, it is the $m=1$ case of Theorem \ref{thm discrete higher}. We single out this case here in order to compare with Theorem \ref{thm chakerian 2}.

\begin{theorem}[Discrete Chakerian's isoperimetric inequality]\label{thm1}
Let $k\ge 3$.
For any $k$-gon $P$ with centroid $0$,
\begin{equation*}
S (P)- 4 \tan \left(\frac{\pi}{k}\right) F(P) \ge 2\tan^2\left(\frac{\pi}{k}\right) \left\|z+\frac{i e^{-i \frac{\pi}{k}} \boldsymbol{t}}{2 \sin \left(\frac{\pi}{k}\right)}\right\|^{2}.
\end{equation*}
The equality holds if and only if $P$ has the form $\zeta_1R_1+\zeta_{k-1}R_{k-1}$.
\end{theorem}

\begin{proof}

This follows from Proposition \ref{prop4} and the discrete Wirtinger inequality $\|\dot{z}\|^2\ge 4 \sin^2\left(\frac{\pi}{k}\right)\|z\|^2$ given that $z$ has centroid $0$.

If the equality holds, by the equality case of Proposition \ref{prop1}, we have $z=\zeta_1 R_1+\zeta_{k-1}R_{k-1}$.
\end{proof}

\begin{corollary}
For an equilateral $k$-gon $P$ with perimeter $L, L^{2} \ge 4 k \tan \left(\frac{\pi}{k}\right)|F(P)|$, and equality holds if and only if $P$ is a regular $k$-gon.
\end{corollary}

If we are allowed to weaken the lower bound in Theorem \ref{thm1}, we can obtain a similar result with a better rigidity case.
\begin{theorem} [Discrete Chakerian's isoperimetric inequality]\label{thm chakerian 2}
Let $k \ge 3$. For any $k$-gon $P$ with centroid 0,
$$
S(P)-4 \tan \left(\frac{\pi}{k}\right) F(P) \ge 2 \sin^{2}\left(\frac{\pi}{k}\right)\left\|z+\frac{i e^{-i \frac{\pi}{k}} \boldsymbol{t}}{2 \sin \left(\frac{\pi}{k}\right)}\right\|^{2}.
$$
The equality holds if and only if $P$ is a positively oriented regular $k$-gon.
\end{theorem}

\begin{proof}
By \eqref{pf1} and the discrete Wirtinger inequality $\|\dot{z}\|^{2} \ge 4 \sin^{2}\left(\frac{\pi}{k}\right)\|z\|^{2}$,
\begin{equation*}
\begin{split}
\left\|\dot{z}-2 i \sin \left(\frac{\pi}{k}\right) e^{i \frac{i}{k}} z\right\|^{2}
\le& 2\|\dot z\|^{2}-8 \sin \left(\frac{\pi}{k}\right)\left[-\frac{1}{2} \sum_{\nu=1}^{k-1}\left|\zeta_{\nu}\right|^{2} \mathrm{Im}\left\langle e^{i \frac{\pi}{k}} R_{\nu}, \dot{R}_{\nu}\right\rangle\right]\\
=& 2\|\dot z\|^{2}-8 \sin \left(\frac{\pi}{k}\right)\left[\sum_{\nu=1}^{k-1}|\zeta_\nu|^2\cdot k \sin \left(\frac{\nu \pi}{k}\right) \cos \left(\frac{(\nu-1) \pi}{k}\right)\right].
\end{split}
\end{equation*}
As $F(R_\nu)=k \sin \left(\frac{\nu \pi}{k}\right) \cos \left(\frac{\nu \pi}{k}\right)$, we have $\mathrm{Im}\left\langle e^{i\frac{\pi}{k}}R_{\nu}, \dot{R_{\nu}}\right\rangle=\frac{\cos \left(\frac{(\nu-1) \pi}{k}\right)}{\cos \frac{\nu \pi}{k}} F\left(R_{\nu}\right)$ unless $\cos \frac{\nu \pi}{k}=0$.
In the case where $\cos \frac{\nu \pi}{k}=0$, $k=2\nu$ and so $-\frac{1}{2}\mathrm{Im}\langle e^{i\frac{\pi}{k}}R_\nu, \dot{R_\nu}\rangle =k\sin \left(\frac{\pi}{k}\right)>0$ and $F(R_\nu)=0$.
By combining with the $m=1$ case of Proposition \ref{prop1}, we arrive at
\begin{equation}\label{ineq1}
\begin{split}
\left\|\dot{z}-2 i \sin \frac{\pi}{k} e^{i \frac{\pi}{k}} z\right\|^{2} \le 2\|\dot{z}\|^{2}-8 \sin \left(\frac{\pi}{k}\right) \sum_{0<\nu<k, \nu \ne \frac{k}{2}}\left|\zeta_{\nu}\right|^{2} \frac{\cos \left(\frac{(\nu-1) \pi}{k}\right)}{\cos \left(\frac{\nu \pi}{k}\right)} F\left(R_{\nu}\right).
\end{split}
\end{equation}

If $\nu \le \left\lfloor\frac{k}{2}\right\rfloor$ and $\nu\ne \frac{k}{2}$, then $F(R_\nu)\ge 0$ and
$$\frac{\cos \frac{(\nu-1) \pi}{k}}{\cos \left(\frac{\nu \pi}{k}\right)}=\cos \left(\frac{\pi}{k}\right)+\tan \left(\frac{\nu \pi}{k}\right) \sin \left(\frac{\pi}{k}\right) \ge \cos \left(\frac{\pi}{k}\right)+ \tan \left(\frac{\pi}{k}\right) \sin \left(\frac{\pi}{k}\right)=\frac{1}{\cos \left(\frac{\pi}{k}\right)}. $$

If $ \nu>\left\lfloor\frac{k}{2}\right\rfloor$, then $F(R_\nu)<0$, and
\begin{equation}\label{ineq3}
\begin{split}
-\frac{\cos \left(\frac{(\nu-1) \pi}{k}\right)}{\cos \left(\frac{\nu \pi}{k}\right)} F\left(R_{\nu}\right) & =\frac{\cos \left(\frac{(\nu-1) \pi}{k}\right)}{\cos \left(\frac{\nu \pi}{k}\right)}\left|F\left(R_{\nu}\right)\right| \\
& =\left(\cos \left(\frac{\pi}{k}\right)+\tan \left(\frac{\nu \pi}{k}\right) \sin \left(\frac{\pi}{k}\right)\right)\left|F\left(R_{\nu}\right)\right| \\
& \le \left(\cos \left(\frac{\pi}{k}\right)+\tan \left(\frac{\pi}{k}\right) \sin \left(\frac{\pi}{k}\right)\right)\left|F\left(R_{\nu}\right)\right| \\
& =-\frac{1}{\cos \left(\frac{\pi}{k}\right)} F\left(R_{\nu}\right).
\end{split}
\end{equation}

In view of \eqref{ineq1}, we have
\begin{equation*}
\begin{split}
\left\|\dot{z}-2 i \sin \left(\frac{\pi}{k}\right) e^{i \frac{\pi}{k}} z\right\|^{2} & \le 2\|\dot{z}\|^{2}-8 \sin \left(\frac{\pi}{k}\right) \sum_{\substack{0<\nu<k\\
\nu \ne \frac{k}{2}}}\left|\zeta_{\nu}\right|^{2} \frac{\cos \left(\frac{(\nu-1) \pi}{k}\right)}{\cos \left(\frac{\nu \pi}{k}\right)} F\left(R_{\nu}\right) \\
& \le 2\left\|\dot{z}\right\|^{2}-8 \tan \left(\frac{\pi}{k}\right) \sum_{\nu=1}^{k-1}\left|\zeta_{\nu}\right|^{2} F\left(R_{\nu}\right) \\
& =2\left(S(P)-4 \tan \left(\frac{\pi}{k}\right) F(P)\right).
\end{split}
\end{equation*}

If the equality holds, by the equality case of Proposition \ref{prop1}, we have $z=\zeta_1 R_1+\zeta_{k-1}R_{k-1}$.
When $\nu=k-1$, we have $\nu>\left\lfloor \frac{k}{2}\right\rfloor$ and the inequality in \eqref{ineq3} is strict. Therefore \eqref{ineq1} is strict unless $\zeta_{k-1}=0$.
Hence $z$ represents a regular convex $k$-gon.
\end{proof}

In \cite{kwong2021higher}, the authors proved that for a smooth closed curve $C$ on the plane whose centroid is $0$,
\begin{equation}\label{KL}
\int_{C}\left|z-\left(\frac{L}{2 \pi}\right) \boldsymbol{n}\right|^{2} d s+\frac{1}{3} \int_{C}\left|z+\left(\frac{L}{2 \pi}\right)^{2} \boldsymbol{\kappa}\right|^{2} d s-\frac{L}{2 \pi^{2}}\left(L^{2}-4 \pi F\right) \ge 0.
\end{equation}
Here $\boldsymbol{n}$ is the unit outward normal, $z$ is the position vector, $L$ is the length of $C $, $F$ is the area enclosed by $C$ and $\boldsymbol{\kappa}$ is the (smooth) curvature vector.

In the following, we are going to give the discrete analogue of this result. Let us remark that for a regular $k$-gon whose centroid is $0$, the curvature vector is equal to $-4\sin^2\left(\frac{\pi}{k}\right)z$, and so the term $ \boldsymbol{\kappa}+4 \sin^{2}\left(\frac{\pi}{k}\right) z $ is natural in the following result.

\begin{theorem}\label{thm discrete higher}
Let $m \le \left\lfloor\frac{k}{2}\right\rfloor$ be a natural number. Suppose $P$ is a $k$-gon represented by $z$ whose centroid is $0$. Then
\begin{equation*}
\begin{split}
0\le& 2 S_{m, 0} \cos^{2}\left(\frac{\pi}{k}\right)\left(S(P)-4\tan \left(\frac{\pi}{k}\right) F(P)\right)-S_{m, 0}\left\|\boldsymbol{t}-2 i \sin \left(\frac{\pi}{k}\right) e^{i \frac{\pi}{k}} z\right\|^{2} \\
& +\sum_{l=1}^{m-1} S_{m, l}\left\|D^{l-1}\left(\boldsymbol{\kappa}+4 \sin^{2}\left(\frac{\pi}{k}\right) z\right)\right\|^{2}.
\end{split}
\end{equation*}
Here $S_{m, l}$ are defined in Proposition \ref{prop s}.
The equality holds if and only if the Fourier coefficients of $z$ satisfy $\zeta_{\nu}=0$ for $m<\nu<k-m$.
\end{theorem}
\begin{proof}
This follows from Proposition \ref{prop s} and Proposition \ref{prop4}.
\end{proof}
\begin{remark}
For $m \le \left\lfloor\frac{k}{2}\right\rfloor$, note that $S_{m, 0}=4^{m-1}\prod_{j=2}^{m}\left(\sin^{2}\left(\frac{\pi}{k}\right)-\sin^{2}\left(\frac{j\pi}{k}\right)\right)$ is positive when $m$ is odd and negative when $m$ is even.
So in the case where $m$ is even, Theorem \ref{thm discrete higher} gives the following upper bound:
\begin{align*}
& 2 \cos^{2}\left(\frac{\pi}{k}\right)\left(S(P)-4 \tan \left(\frac{\pi}{k}\right) F(P)\right) \\
\le & \left\|\boldsymbol{t}-2 i \sin \left(\frac{\pi}{k}\right) e^{i \frac{\pi}{k}} z\right\|^{2}-\widetilde{S}_{m, 0}^{-1}\sum_{l=1}^{m-1} S_{m, l}\left\|D^{l-1}\left(\boldsymbol{\kappa}+4 \sin^{2}\left(\frac{\pi}{k}\right) z\right)\right\|^{2}.
\end{align*}
In the case where $m$ is odd, Theorem \ref{thm discrete higher} gives the following lower bound:
\begin{align*}
& 2 \cos^{2}\left(\frac{\pi}{k}\right)\left(S(P)-4 \tan \left(\frac{\pi}{k}\right) F(P)\right) \\
\ge & \left\|\boldsymbol{t}-2 i \sin \left(\frac{\pi}{k}\right) e^{i \frac{\pi}{k}} z\right\|^{2}- {\widetilde S_{m, 0}}^{-1} \sum_{l=1}^{m-1} S_{m, l}\left\|D^{l-1}\left(\boldsymbol{\kappa}+4 \sin^{2}\left(\frac{\pi}{k}\right) z\right)\right\|^{2}.
\end{align*}
Here $S_{m, l}$ are defined in Proposition \ref{prop s} and $\widetilde S_{m, 0}=|S_{m, 0}|=4^{m-1} \prod_{j=2}^{m}\left(\sin^{2}\left(\frac{j\pi}{k}\right)-\sin^{2}\left(\frac{ \pi}{k}\right)\right)$.
\end{remark}

\begin{corollary}
Let $m < \frac{k}{2} $ be an even natural number. Suppose $P$ is a $k$-gon represented by $z$ whose centroid is $0$. Then
\begin{equation*}
\begin{split}
& 2\cos^{2}\left(\frac{\pi}{k}\right)\left(L(P)^2-4 k\tan \left(\frac{\pi}{k}\right) F(P)\right) \\
\le & k \left\|\boldsymbol{t}-2 i \sin \left(\frac{\pi}{k}\right) e^{i \frac{\pi}{k}} z\right\|^{2}- k\widetilde{S}_{m, 0}^{-1} \sum_{l=1}^{m-1} S_{m, l}\left\|D^{l-1}\left(\boldsymbol\kappa +4 \sin^{2}\left(\frac{\pi}{k}\right) z\right)\right\|^{2}.
\end{split}
\end{equation*}

Here $\widetilde{S}_{m, 0}=\left|S_{m, 0}\right|=4^{m-1} \prod_{j=2}^{m}\left(\sin^{2}\left(\frac{j \pi}{k}\right)-\sin^{2}\left(\frac{\pi}{k}\right)\right)$.

The equality holds if and only if $ P$ is represented by $z=\zeta_\nu R_\nu$ for some $0<|\nu|\le m$.
\end{corollary}
\begin{proof}
The inequality follows from Theorem \ref{thm discrete higher} and the Cauchy-Schwarz inequality $kS(P)\ge L^2(P)$. By the rigidity case of Theorem \ref{thm discrete higher}, the Fourier coefficients of $\dot z=\sum_{\nu=1}^{k-1}\xi_\nu R_\nu$ satisfy $\xi_{\nu}=0$ for $m<\nu<k-m$.
By the equality case of the Cauchy-Schwarz inequality, $\left|\dot{z}_{\nu}\right|^{2}$ is a constant which is independent of $\nu$. Since
$\left(R_{n}\right)_\nu=e^{\frac{2 \pi i n \nu}{k} }$, we have
\begin{equation*}
\left|\dot{z}_{\nu}\right|^{2}=\sum_{n=-2 m}^{2 m}\left(\sum_{n_{1}-n_{2}=n} \xi_{n_{1}} \overline {\xi}_{n_{2}} \right) e^{\frac{2 \pi i n \nu}{k}} =\sum_{n=-2m}^{2m} c_{n} \omega_\nu^{n}
=\text { constant }
\end{equation*}
for all $\nu$, where we have identified $\xi_{k-l}$ with $\xi_{-l}$. By the uniqueness of Fourier series, for $n\ne 0$, we have
$c_n=\sum_{n_{1}-n_{2}=n} \xi_{n_{1}} \bar{\xi}_{n_{2}}=0$. By Lemma \ref{lem nonzero} below, at most one of the coefficients, say $\xi_\nu$, can be non-zero. As $\dot R_\nu=(\omega_\nu-1)R_\nu$ and $z$ has centroid zero, we conclude that $z=\frac{\xi_\nu}{\omega_\nu-1}R_\nu$.
\end{proof}

\begin{lemma}\label{lem nonzero}
Let $m\in \mathbb N$ and suppose $ a_{-m}, a_{-m+1}, \cdots, a_{-1}, a_{1}, \cdots, a_m $ are complex numbers such that for all $n\ne 0$,
$c_n=\sum_{\substack {p-q=n \\
0<|p|, |q|\le m}} a_p \overline {a_{q}}=0
$.
Then at most one of the $a_p$ is non-zero.
\end{lemma}
\begin{proof}
We prove it by induction on $m$. The case where $m=1$ is easy. Assuming the induction hypothesis, we now fix a natural number $m$.

As $c_{2m}=a_m\overline {a_{-m}}=0$, either $a_m=0$ or $a_{-m}=0$. By interchanging $a_{l}$ with $a_{-l}$ if necessary, we may assume $a_{-m}=0$. We can assume $a_{m}\ne0$, for otherwise it reduces to the $(m-1)$-case and we can apply the induction hypothesis, and then the results follows.

We claim that $a_m$ is the only nonzero coefficient.

As $a_{-m}=0$, we have $c_{2m-1}=a_m\overline {a_{-m+1}}=0$. The assumption $a_m\ne0$ then implies $a_{-m+1}=0$. We can go on to consider $c_{2m-2}, c_{2m-3}, \cdots, c_{1}$ one by one to deduce that the remaining coefficients $a_{-m+2}=a_{-m+3}=\cdots=a_{m-1}=0$. This proves the claim.
\end{proof}

\begin{remark}\label{rmk2}
Let us examine the inequalities in Theorem \ref{thm discrete higher} for $m=1$ and $2$.
\begin{enumerate}
\item
As explained before, when $m=1$, the inequality becomes
\begin{align*}
2 \cos^{2}\left(\frac{\pi}{k}\right)\left(S(P)-4 \tan \left(\frac{\pi}{k}\right) F(P)\right)\ge\left\|\boldsymbol{t}-2 i \sin \left(\frac{\pi}{k}\right) e^{i \frac{\pi}{k}} z\right\|^{2}.
\end{align*}
This is the discrete version of Chakerian's sharpened isoperimetric inequality.

\item
When $m=2$ and hence $k\ge 4$, the inequality becomes
\begin{align*}
& 8\left(\sin^{2} \left(\frac{2 \pi}{k}\right)-\sin^{2} \left(\frac{\pi}{k}\right)\right) \cos^{2}\left(\frac{\pi}{k}\right)\left(S(P)-4 \tan \left(\frac{\pi}{k}\right) F(P)\right)\\
\le& 4\left(\sin^{2}\left(\frac{2 \pi}{k}\right)-\sin^{2}\left(\frac{\pi}{k}\right)\right) \left\|\boldsymbol{t}-2 i \sin \left(\frac{\pi}{k}\right) e^{i \frac{\pi}{k}} z\right\|^{2}+\left\| \boldsymbol{\kappa}+4 \sin^{2}\left(\frac{\pi}{k}\right) z \right\|^{2}.
\end{align*}

This is the discrete analogue of \eqref{KL}.
By Proposition \ref{prop stab}, it can also be rearranged so that it gives a measure of the deficit of the discrete Chakerian's isoperimetric inequality:
\begin{equation*}\label{case m=2}
\begin{split}
& 2 \cos^{2}\left(\frac{\pi}{k}\right)\left(S(P)-4 \tan \left(\frac{\pi}{k}\right) F(P)\right)-\left\|\boldsymbol{t}-2 i \sin \left(\frac{\pi}{k}\right) e^{i \frac{\pi}{k}} z\right\|^{2} \\
\le & \frac{1}{4 \sin^{2}\left(\frac{2 \pi}{k}\right)}\left[4 \sin^{2}\left(\frac{\pi}{k}\right)\left(\|\dot{z}\|^{2}-4 \sin^{2}\left(\frac{\pi}{k}\right)\|z\|^{2}\right)+\left\|\boldsymbol{\kappa}+4 \sin^{2}\left(\frac{\pi}{k}\right) z\right\|^{2}\right]\\
=& \frac{1}{4 \sin^{2}\left(\frac{2 \pi}{k}\right)}\left[4 \sin^{2}\left(\frac{\pi}{k}\right)
\left(2 \cos^{2}\left(\frac{\pi}{k}\right)\left(S(P)-4 \tan \left(\frac{\pi}{k}\right) F(P)\right)-\left\|\boldsymbol{t}-2 i \sin \left(\frac{\pi}{k}\right) e^{i \frac{\pi}{k}} z\right\|^{2}\right)\right. \\
& \left. +\left\|\boldsymbol\kappa+4 \sin^{2}\left(\frac{\pi}{k}\right) z\right\|^{2}\right].
\end{split}
\end{equation*}
\item
Similar conclusion holds for $m=3$ and beyond. Please refer to the arXiv version for details.
\end{enumerate}
\end{remark}

\section{The smooth case}\label{sec smooth}

We now give the result analogous to Theorem \ref{thm discrete higher} for a smooth curve. In some sense, it is easier because the geometric quantities are more straightforward to compute and the higher order Wirtinger inequalities are somewhat simpler.

First we need the smooth version of Proposition \ref{prop s}.

\begin{proposition}[\cite{kwong2021higher} Proposition 1]\label{prop higher}
Let $m \in \mathbb N$ and $z: \mathbb R \rightarrow \mathbb C$ be a $2 \pi$-periodic function of class $C^{m}$ with zero mean.Then
\begin{equation}\label{gen ineq}
0 \le \prod_{j=2}^{m}\left(1-j^{2}\right) \int_{0}^{2 \pi}\left(|\dot{z}|^{2}-|z|^{2}\right) d t+\sum_{l=1}^{m-1} s_{m, l} \int_{0}^{2 \pi}\left|z^{(l+1)}+z^{(l-1)}\right|^{2} d t.
\end{equation}

Here, the constants $s_{m, 1}, \cdots, s_{m, m-1}$ are the coefficients of the $(m-2)$-th degree polynomial $\mathrm {S}_{m}(t) \in \mathbb Z [t]$, defined by
$ \mathrm{S}_{m}(t):=\frac{ \mathrm {P}_{m}(t)-\mathrm {P}_{m}(1)}{t-1}=\sum_{l=1}^{m-1} s_{m, l} t^{l-1}$ and $\mathrm{P}_{m}(t):=\prod_{j=2}^{m}\left(t-j^{2}\right)=\left(t-2^{2}\right) \cdots\left(t-m^{2}\right)$.

The equality holds if and only if the function $z$ is of the form
$$
z(t)=\sum_{1\le |n|\le m}a_{n} e^{int}
$$
for some constants $a_{n} \in \mathbb C$.
\end{proposition}

\begin{theorem}\label{smooth thm}
Let $C$ be a simple closed $C^{m}$ curve in $\mathbb C$ with length $L$ and a unit speed counter-clockwise parametrization $z (s) $ with $s \in[0, L] $.
Assume that $C =\partial \Omega$ is the boundary curve of a domain $\Omega \subset \mathbb C$ with the area $F$. Let $\boldsymbol{n} $ denotes the outward-pointing unit normal on $C$. Assume $C$ has centroid $0$, then
\begin{equation*}
\begin{split}
0 \le & \sum_{l=1}^{m-1} s_{m, l}\left(\frac{L}{2 \pi}\right)^{2 l-3} \int_{C}\left|\left(\frac{d}{d s}\right)^{l-1}\left(z+\left(\frac{L}{2 \pi}\right)^{2} \boldsymbol{\kappa}\right)\right|^{2} d s \\
& -\frac{(-1)^{m}}{2}(m-1) !(m+1) !\left(\frac{1}{\pi}\left(L^{2}-4 \pi F\right)-\frac{2 \pi}{L} \int_{C}\left|z-\left(\frac{L}{2 \pi}\right) \boldsymbol{n}\right|^{2} d s\right).
\end{split}
\end{equation*}
The equality holds if and only if $z(t)=be^{it}$ for some $0\ne b\in \mathbb C$.
\end{theorem}

\begin{proof}
Using the change of variable $t=\frac{2\pi}{L}s$, we have
$$
z(t)+\ddot{z}(t)=z+\left(\frac{L}{2 \pi}\right)^{2} \boldsymbol{\kappa}.
$$
Therefore
\begin{equation}\label{4}
z^{(l+1)}(t)+z^{(l-1)}(t)=\left(\frac{L}{2 \pi}\right)^{l-1} \left(\frac{d}{d s}\right)^{l-1}\left(z+\left(\frac{L}{2 \pi}\right)^{2} \boldsymbol{\kappa}\right).
\end{equation}
It is easy to see that
$$
\frac{L^{2}}{2 \pi}=\int_{0}^{2 \pi}|\dot{z}|^{2} d t
$$
and
$$
2 F=\iint_{\Omega} \mathrm{div} \left(z\right)\, dxdy=-\int_{0}^{2 \pi} \mathrm{Im}(z \dot{\bar{z}}) d t.
$$
So we can write
$$
\frac{1}{\pi}\left(L^{2}-4 \pi F\right)=\int_{0}^{2 \pi}2\left(|\dot{z}|^{2}+\mathrm{Im}(z \dot{\bar{z}})\right) d t.
$$
We rewrite this integrand as
$$
2\left(|\dot{z}|^{2}+\mathrm{Im}(z \dot{\bar{z}})\right)=|z+i \dot{z}|^{2}+|\dot{z}|^{2}-|z|^{2}.
$$
The first term on the right has integral
\begin{equation*}
\int_{0}^{2\pi}|z+i \dot{z}|^2dt=\frac{2 \pi}{L} \int_{C}\left|z-\left(\frac{L}{2 \pi}\right) \boldsymbol{n}\right|^{2} d s.
\end{equation*}
Therefore
\begin{equation}\label{5}
\int_{0}^{2\pi}\left(|\dot z|^{2}-|z|^{2}\right)dt
=\frac{1}{\pi}\left(L^{2}-4 \pi F\right)-\frac{2 \pi}{L} \int_C\left|z-\left(\frac{L}{2 \pi}\right) \boldsymbol{n}\right|^{2}ds.
\end{equation}

Putting \eqref{4}, \eqref{5} into \eqref{gen ineq}, we have
\begin{equation*}
\begin{split}
0\le& \sum_{l=1}^{m-1} s_{m, l} \int_{0}^{2 \pi}\left|z^{(l+1)}+z^{(l-1)}\right|^{2} d t -\frac{(-1)^{m}}{2}(m-1) !(m+1) ! \int_{0}^{2 \pi}\left(|\dot{z}|^{2}-|z|^{2}\right) d t\\
=& \sum_{l=1}^{m-1} s_{m, l} \left(\frac{L}{2 \pi}\right)^{2l-3} \int_{C}\left|\left(\frac{d}{d s } \right)^{l-1}\left(z+\left(\frac{L}{2\pi}\right)^2\boldsymbol{\kappa}\right)\right|^{2} d s \\
& -\frac{(-1)^{m}}{2}(m-1) !(m+1) !
\left(\frac{1}{\pi}\left(L^{2}-4 \pi F\right)-\frac{2 \pi}{L} \int_{C}\left|z-\left(\frac{L}{2 \pi}\right) \boldsymbol{n}\right|^{2} d s\right).
\end{split}
\end{equation*}
Suppose that the equality holds. By Proposition \ref{prop higher}, we have
$$
\dot z(t)=\sum_{0<|n| \le m} a_{n} e^{i n t}
$$
for some constants $a_n\in \mathbb C$. Also, , $|\dot z(t)|^{2}=\mu$ is a constant. Then
$$
\mu=\dot z(t) \overline{\dot z(t)}=\sum_{n=-2m}^{2m}\left(\sum_{p+q=n} a_{p} \overline{a_{-q}}\right) e^{i n t}
$$
The uniqueness of Fourier series shows that, for all $0<|n| \le 2m$, the coefficients vanish:
$$
\sum_{p-q=n, 1 \le |p|, |q| \le 2} a_{p} \overline{a_{q}}=0
$$
By Lemma \ref {lem nonzero}, at most one of the coefficients $a_n$ is nonzero, i.e. $\dot z(t)=a_n e^{int}$ for some $0<|n| \le m$.
As $z(t)$ has centroid $0$, we have $z(t)= \frac{a_n}{in}e^{int}$. Since $z(t)$ is assumed to be simple and is positively oriented, we deduce that $z(t)=b e^{it}$ for some $0\ne b\in \mathbb C$.
\end{proof}

\section{Higher order Chernoff inequality}\label{chernoff}
We now prove a generalization of the Chernoff inequality \cite{chernoff1969area}, which states that
for a closed convex curve $\gamma$ on $\mathbb R^{2}$ with area $F$ and width function $w(\theta)$, we have
$$
F \le \frac{1}{2} \int_{0}^{\pi / 2} w(\theta) w\left(\theta+\frac{1}{2} \pi\right) d \theta.
$$
The equality holds if and only if $\gamma$ is a circle.
Ou and Pan \cite{ou2010some} proved the following generalized version of the Chernoff inequality:
$$
F \le \frac{1}{k} \int_{0}^{\pi / k} w_{k}(\theta) w_{k}\left(\theta+\frac{1}{k} \pi\right) d \theta
$$
where $w_{k}(\theta)=h(\theta)+h\left(\theta+\frac{2 \pi}{k}\right)+\cdots+h\left(\theta+\left(\frac{2(k-1) \pi}{k}\right)\right)$ and $h$ is the support function.

The definition of $w_k$ suggests we consider the following operator: let $T_k$ be the transform defined on the space of real-valued functions on $\mathbb S^{1}$, given by
\begin{align*}
T_k[h](\theta):=\frac{1}{k}\sum_{m=1}^{k} h\left(\theta+\frac{(2m-1)\pi}{k} \right).
\end{align*}
If $h$ is the support function of a convex curve, then the convex curve with support $T_k[h]$ is one whose support is the average of $h$ in $k$ directions.
Now, let $\gamma$ be a simple closed convex curve parametrised by the normal angle $\theta$, and $h=h(\theta)$ be the support function of $\gamma$.
Let $\displaystyle h=\sum_{n=-\infty}^{\infty}a_ne^{in\theta}$ be the Fourier series of $\displaystyle h$. Then the Fourier series of $\displaystyle T_k[h]$ becomes
\begin{equation*}
\begin{split}
T_k[h](\theta)=\sum_{n=-\infty}^{\infty} a_{n} \frac{1}{k}\sum_{m=1}^{k}e^{ i n\left(\theta+\frac{(2 m-1)\pi}{k} \right)}
=& \sum_{n=-\infty}^{\infty} a_{n} \frac{1}{k}\sum_{m=1}^{k}e^{ \frac{in(2 m-1)\pi}{k} }e^{in\theta}\\
=& \sum_{n=-\infty}^{\infty} a_{n} \frac{1}{k} \sum_{m=1}^{k} \cos \frac{n(2 m-1) \pi}{k}e^{in \theta}\\
=& \sum_{n=-\infty}^{\infty} a_{n} \beta_n e^{in \theta},
\end{split}
\end{equation*}
where $\displaystyle \beta_{n}=\frac{1}{k} \sum_{m=1}^{k} \cos \frac{n(2 m-1) \pi}{k}$. It is easy to see that $\displaystyle \beta_{n} \in\{ -1, 0, 1\} $, $\beta_0=1$ and $\beta_{1}=0$ if $k\ge 2$. Indeed, $\beta_n\ne 0$ if and only if $n$ is a multiple of $k$.

\begin{equation*}
\begin{split}
\end{split}
\end{equation*}
\begin{align*}
\end{align*}
On the other hand, consider the operator
$A[h]=h+\ddot h$.
The Fourier series of $A[h]$ is given by
\begin{align*}
A[h](\theta)=\sum_{n=-\infty}^{\infty} a_{n}(1-n^2) e^{i n \theta}
=\sum_{n=-\infty}^{\infty} a_{n} \delta_{n} e^{i n \theta}.
\end{align*}
Obviously, $\beta_n\ge \delta_n$ for all $n$.
As shown in Section \ref{sec idea}, this shows that for $k\ge 2$ and $m\in \mathbb N$, we have
\begin{align}\label{ineq gen chernoff}
\int_{0}^{2\pi} h(\theta)\cdot(T_k-A)^m[h](\theta)d\theta\ge0.
\end{align}
Since $\beta_n>\delta_n$ for $|n|\ge 2$, it follows that the equality case holds if and only if $h=\sum_{n=-1}^{1}a_ne^{in\theta}$, with $a_{-n}=\overline{a_{n}}$. This together with the formula
$\gamma(\theta)=h(\theta) \boldsymbol n +\dot{h}(\theta) \boldsymbol t
=h(\theta) e^{i\theta} +\dot{h}(\theta) i e^{i\theta}=2\overline a_1+a_0e^{i\theta}$, shows that $\gamma$ is a circle.

It remains to give a geometric interpretation to the inequality \eqref{ineq gen chernoff}.
Recall the definition $w_{k}(\theta)=h(\theta)+h\left(\theta+\frac{2 \pi}{k}\right)+\cdots+h\left(\theta+\left(\frac{2(k-1) \pi}{k}\right)\right)$.
When $m=1$, we have (cf. \cite[Eqn. 3-2]{ou2010some})
\begin{equation}\label{ou}
\begin{split}
\frac{1}{k} \int_{0}^{\frac{\pi}{k}} w_{k}(\theta) w_{k}\left(\theta+\frac{\pi}{k}\right) d \theta
=& \frac{1}{2 k} \sum_{m=1}^{k} \int_{0}^{2 \pi} h(\theta) h\left(\theta+\frac{(2 m-1) \pi}{k}\right) d \theta\\
=& \frac{1}{2} \int_{0}^{2 \pi} h(\theta) T_{k}[h](\theta) d \theta.
\end{split}
\end{equation}

On the other hand, the area $F$ enclosed by $\gamma$ is given by (\cite[Eqn 2.4.27]{Groemer1996})
\begin{align*}
F=\frac{1}{2} \int_{0}^{2 \pi}\left(h(\theta)^2-\dot h(\theta)^2\right) d \theta
=\frac{1}{2} \int_{0}^{2 \pi}h(\theta)\cdot A[h](\theta)d\theta.
\end{align*}
Therefore when $m=1$, we recover the Ou-Pan generalized Chernoff inequality
$$
F \le \frac{1}{k} \int_{0}^{\frac{\pi}{k}} w_{k}(\theta) w_{k}\left(\theta+\frac{\pi}{k} \right) d \theta.
$$
Now we assume $m\ge2$.
We can parameterise $\gamma$ by the normal angle $\theta$. In fact, it can be explicitly parametrised by \cite[p. 34]{ChouZhu2001}
\begin{align*}
\gamma(\theta)=h(\theta)\boldsymbol n +\dot h(\theta)\boldsymbol t,
\end{align*}
where $\boldsymbol n=(\cos \theta, \sin \theta)$ and $\boldsymbol t=(-\sin \theta, \cos \theta)$.
The locus of curvature centers $\gamma_{(1)}$ of $\gamma$ is then defined to be $\gamma(\theta)-\rho (\theta) \boldsymbol n $, where $\rho$ is the radius of curvature. Since $\rho=h+\ddot h$,
\begin{align*}
\gamma_{(1)}(\theta)=\dot{h} \boldsymbol t-\ddot{h} \boldsymbol n.
\end{align*}
Of course, this curve can fail to be simple and convex, but we can still regard it as a smooth parametrised curve and compute its algebraic area using the formula $\mathrm{Area}=\frac{1}{2}\int xdy-ydx$.
Differentiating $\gamma_{(1)}$ gives $\dot \gamma_{(1)}=-(\dot{h}+\dddot{h}) \boldsymbol n$, and hence we can fix $\boldsymbol t$ as a unit normal field of $\gamma_{(1)}$, and w.r.t. this choice of normal, the support function of $\gamma_{(1)}$ is then $\dot h$.

Inductively, we can define the locus of curvature centers $\gamma_{(2)}$ of $\gamma_{(1)}$, and so on, and deduce that $\gamma_{(j)}$ has support function $h^{(j)}(\theta)$ w.r.t. the normal $ J^j \boldsymbol n $, where $J$ is the anti-clockwise rotation by $\frac{\pi}{2}$. Therefore the algebraic area $F$ of $\gamma_{(j)}$ is given by
\begin{align}\label{alg area}
F[\gamma_{(j)}]=\frac{1}{2}\int_{0}^{2\pi} ({h^{(j)}(\theta)}^2-{h^{(j+1)}(\theta)}^2)d\theta.
\end{align}

The R.H.S. of \eqref{ineq gen chernoff} can be written as
\begin{equation}\label{R.H.S. }
\begin{split}
& \sum_{j=0}^{m}(-1)^{m-j} \binom{m }{j}\int_{0}^{2\pi}h \cdot A^{m-j} \left(T_{k}\right)^{j}[h]d\theta\\
=& (-1)^{m}\int_{0}^{2\pi}h \cdot A^{m} [h]d\theta+\int_{0}^{2 \pi} h \cdot\left(T_{k}\right)^{m}[h] d \theta+
\sum_{j=1}^{m-1}(-1)^{m-j} \binom{m }{j}\int_{0}^{2\pi}h \cdot A^{m-j} \left(T_{k}\right)^{j}[h]d\theta.
\end{split}
\end{equation}

Using \eqref{alg area}, it is not hard to show by induction that
\begin{equation}\label{higher alg area}
\begin{split}
\int_{0}^{2\pi}h A^{m}[h]
=& \sum_{r=0}^{m-1}(-1)^r\binom{m-1}{r} \int_{0}^{2 \pi}\left(h^{(r)}(\theta)^{2}-h^{(r+1)}(\theta)^{2}\right) d \theta\\
=& 2\sum_{r=0}^{m-1}(-1)^{r}\binom{m-1}{r} F\left[\gamma_{(r)}\right].
\end{split}
\end{equation}
It is also easy to see that $\left(T_k\right)^{l+2}=\left(T_k\right)^{l}$ if $l\ge 1$. So for $j\ge 1$, the term $\left(T_{k}\right)^{j}[h]$ only depends on $j\pmod 2$. By abuse of notation, we denote by $T_k\gamma$ the convex curve whose support function is $T_k[h]$ and by $T_k^2\gamma$ the convex curve whose support function is $(T_k)^2[h]$.

We now consider the term $\int_{0}^{2 \pi} h \cdot\left(T_{k}\right)^{m}[h] d \theta$ in \eqref{R.H.S. }. If $m$ is odd, then \eqref{ou} gives
\begin{align*}
\int_{0}^{2 \pi} h(\theta)(T_{k})^{m}[h](\theta) d \theta
=\int_{0}^{2 \pi} h(\theta) T_{k} [h](\theta) d \theta
=\frac{2}{k} \int_{0}^{\frac{\pi}{k}} w_{k}(\theta) w_{k}\left(\theta+\frac{\pi}{k}\right) d \theta.
\end{align*}
When $m$ is even,
\begin{align*}
\int_{0}^{2 \pi} h(\theta)(T_{k})^{m}[h](\theta) d \theta
=& \int_{0}^{2 \pi} h(\theta)(T_{k})^2[h](\theta) d \theta.
\end{align*}
Let us consider the Hermitian form defined on $C(\mathbb S^1, \mathbb C)$ defined by
\begin{align*}
I_1(\phi, \psi)=\int_{0}^{2\pi} \overline \psi \cdot (T_k)^2[\psi]d\theta
=\int_{0}^{2\pi} \overline \psi \cdot\frac{1}{k} \sum_{j=0}^{k-1} \psi\left(\theta+\frac{2 j \pi}{k}\right) d\theta.
\end{align*}
We claim that $\{e^{in\theta}\}$ forms an orthogonal (but not necessarily orthonormal) Schauder basis for $I_1$. To see this, note that $I_1(e^{in\theta}, e^{im\theta})$ is the integral of
\begin{align*}
e^{-im\theta}\cdot \frac{1}{k} \sum_{j=0}^{k-1} e^{in\left(\theta+\frac{2 j \pi}{k}\right)}
=e^{i(n-m)\theta}\cdot \frac{1}{k} \sum_{j=0}^{k-1} e^{ \frac{2i n j \pi}{k} }.
\end{align*}
If $\displaystyle k\nmid n$, then $\displaystyle \sum_{j=0}^{k-1} e^{\frac{2 i n j \pi}{k}}=0$. Otherwise, $\displaystyle k\mid n$ and the integrand is just $\displaystyle e^{i(n-m)\theta}$, which has integral $0$ over $[0, 2\pi]$ when $m\ne n$. So for $\displaystyle \phi=\sum_n a_n e^{in\theta}$, $\displaystyle I_1(\phi, \phi)=2\pi\sum_{n: k|n} |a_n|^2$.

Similar calculation shows that the Hermitian form $\displaystyle I_2(\phi, \psi)=\int_{0}^{\frac{2\pi}{k}} \overline {\psi} \cdot (T_{k})^2 [\phi] d \theta$ satisfies the property that for $ \displaystyle \phi=\sum_n a_n e^{in\theta}$, $ \displaystyle I_2(\phi, \phi)=\frac{2\pi}{k} \sum_{n: k \mid n} |a_n|^2=\frac{1}{k}I_1(\phi, \phi)$.

Therefore, when $m$ is even,
$$
\int_{0}^{2 \pi} h(\theta)\left(T_{k}\right)^{m}[h](\theta) d \theta=\int_{0}^{2 \pi} h(\theta)\left(T_{k}\right)^{2}[h](\theta) d \theta
=k \int_{0}^{\frac{2\pi}{k}} \left(T_k[h]\right)^2d\theta=\frac{1}{k}\int_{0}^{\frac{2\pi}{k}} w_k(\theta)^2d \theta.
$$
Now, consider the last term in \eqref{R.H.S. }
\begin{equation*}
\begin{split}
& \sum_{j=1}^{m-1}(-1)^{m-j}\binom{m}{j} \int_{0}^{2 \pi} h \cdot A^{m-j}\left(T_{k}\right)^{j}[h]\\
=& (-1)^{m} \sum_{1\le j \le m -1\atop 2 \mid j}\binom{m}{j} \int_{0}^{2 \pi} h \cdot A^{m-j} T_{k}^{2}[h]
-(-1)^{m} \sum_{1\le j \le m -1\atop 2 \nmid j}\binom{m}{j} \int_{0}^{2 \pi} h \cdot A^{m-j} T_{k}[h] \\
\end{split}
\end{equation*}
The term $\displaystyle \int_{0}^{2 \pi} h \cdot A^{m-j} T_{k}[h]d\theta$ has a similar geometric meaning as \eqref{higher alg area}.
In fact, same computation gives
\begin{equation*}
\begin{split}
\int_{0}^{2 \pi} h \cdot A^{m-j} T_{k}[h] d \theta
=& \sum_{r=0}^{m-1}(-1)^r\binom{m-1-j}{r} \int_{0}^{2 n}\left(h^{(r)}\left(T_{k}[h]\right)^{(r)}-h^{(r+1)}\left(T_{k}[h]\right)^{(r+1)}\right) d \theta\\
=& 2\sum_{r=0}^{m-1}(-1)^{r}\binom{m-1-j}{r}F[\gamma_{(r)}, (T_k\gamma)_{(r)}].
\end{split}
\end{equation*}
Here $F[\cdot, \cdot]$ is the algebraic mixed area (cf. \cite[Eqn 2.4.28]{Groemer1996}).

Combining the above, we have the following conclusion. When $m$ is odd,
\begin{equation*}
\begin{split}
0 \le& - \sum_{r=0}^{m-1}(-1)^{r}\binom{m-1}{r} F[\gamma_{(r)}] +\frac{1}{k} \int_{0}^{\frac{\pi}{k}} w_{k}(\theta) w_{k}\left(\theta+\frac{\pi}{k}\right) d \theta\\
& - \sum_{1\le j \le m-1 \atop 2 \mid j}\binom{m}{j} \sum_{r=0}^{m-1}(-1)^{r}\binom{m-1-j}{r} F\left[\gamma_{(r)}, \left(T_{k}^{2} \gamma\right)_{(r)}\right] \\
& + \sum_{1\le j \le m -1\atop 2 \nmid j}\binom{m}{j} \sum_{r=0}^{m-1}(-1)^{r}\binom{m-1-j}{r} F\left[\gamma_{(r)}, \left(T_{k} \gamma\right)_{(r)}\right].
\end{split}
\end{equation*}
When $m$ is even,
\begin{equation*}
\begin{split}
0 \le& \sum_{r=0}^{m-1}\binom{m-1}{r}(-1)^{r} F[\gamma_{(r)}] +\frac{1}{2k} \int_{0}^{\frac{2 \pi}{k}} w_{k}(\theta)^{2} d \theta\\
& + \sum_{1\le j \le m-1 \atop 2 \mid j}\binom{m}{j} \sum_{r=0}^{m-1}(-1)^{r}\binom{m-1-j}{r} F\left[\gamma_{(r)}, \left(T_{k}^{2} \gamma\right)_{(r)}\right] \\
& - \sum_{1\le j \le m -1\atop 2 \nmid j}\binom{m}{j} \sum_{r=0}^{m-1}(-1)^{r}\binom{m-1-j}{r} F\left[\gamma_{(r)}, \left(T_{k} \gamma\right)_{(r)}\right].
\end{split}
\end{equation*}

So we have proved the following result.
\begin{theorem}\label{Chernoff}
Let $2\le k\in \mathbb N$, $m\in \mathbb N$ and $\gamma$ be a smooth closed convex curve on $\mathbb R^2$.
When $m$ is odd,
\begin{equation*}
\begin{split}
0 \le& - \sum_{r=0}^{m-1}(-1)^{r}\binom{m-1}{r} F[\gamma_{(r)}] +\frac{1}{k} \int_{0}^{\frac{\pi}{k}} w_{k}(\theta) w_{k}\left(\theta+\frac{\pi}{k}\right) d \theta\\
& - \sum_{1\le j \le m-1 \atop 2 \mid j}\binom{m}{j} \sum_{r=0}^{m-1}(-1)^{r}\binom{m-1-j}{r} F\left[\gamma_{(r)}, \left(T_{k}^{2} \gamma\right)_{(r)}\right] \\
& + \sum_{1\le j \le m -1\atop 2 \nmid j}\binom{m}{j} \sum_{r=0}^{m-1}(-1)^{r}\binom{m-1-j}{r} F\left[\gamma_{(r)}, \left(T_{k} \gamma\right)_{(r)}\right].
\end{split}
\end{equation*}
When $m$ is even,
\begin{equation*}
\begin{split}
0 \le& \sum_{r=0}^{m-1}\binom{m-1}{r}(-1)^{r} F[\gamma_{(r)}] +\frac{1}{2k} \int_{0}^{\frac{2 \pi}{k}} w_{k}(\theta)^{2} d \theta\\
& + \sum_{1\le j \le m-1 \atop 2 \mid j}\binom{m}{j} \sum_{r=0}^{m-1}(-1)^{r}\binom{m-1-j}{r} F\left[\gamma_{(r)}, \left(T_{k}^{2} \gamma\right)_{(r)}\right] \\
& - \sum_{1\le j \le m -1\atop 2 \nmid j}\binom{m}{j} \sum_{r=0}^{m-1}(-1)^{r}\binom{m-1-j}{r} F\left[\gamma_{(r)}, \left(T_{k} \gamma\right)_{(r)}\right].
\end{split}
\end{equation*}
Here $w_{k}(\theta)=\sum_{j=0}^{k-1}h\left(\theta+\frac{2j\pi}{k} \right)$ is the generalized width, $h$ is the support function of $\gamma$, $F$ is the algebraic area or the algebraic mixed area, $(T_k)^j\gamma$ is the curve whose support function is $(T_k)^j[h]$, and $\beta_{(j)}$ is the $j$-th order locus of curvature centers of a curve $\beta$ as defined above.

The equality holds if and only if $\gamma$ is a circle.
\end{theorem}

\end{document}